\documentclass[a4paper,12pt,leqno]{amsart}
\usepackage{amsmath,amssymb,mathrsfs}
\usepackage[matrix,arrow]{xy} 
\usepackage{xcolor,hyperref}
\hypersetup{colorlinks=true,citecolor=black,linkcolor=black}
\newcommand{\h}{\hbox}
\newcommand{\q}{\quad}
\newcommand{\nin}{\noindent}

\newcommand{\ms}{\par\medskip}
\newcommand{\sk}{\par\smallskip}

\newcommand{\msn}{\par\medskip\noindent}
\newcommand{\skn}{\par\smallskip\noindent}
\newcommand{\ges}{\geqslant}
\newcommand{\les}{\leqslant}
\newcommand{\one}{\hskip1pt}

\newcommand{\mopl}{\hbox{$\bigoplus$}}

\newcommand{\Hc}{{\mathcal H}}
\newcommand{\X}{{\mathcal X}}
\newcommand{\PP}{{\mathbb P}}
\newcommand{\Q}{{\mathbb Q}}
\newcommand{\C}{{\mathbb C}}
\newcommand{\DD}{{\mathbb D}}

\newcommand{\RR}{{\mathbf R}}
\newcommand{\Z}{{\mathbb Z}}
\newcommand{\ee}{{\mathbf e}}
\newcommand{\Ht}{\widetilde{H}}

\newcommand{\dfw}{{\rm d}f{\wedge}}
\newcommand{\df}{{\rm def}}
\newcommand{\Ff}{F_{\!f}}

\newcommand{\IC}{{\rm IC}}
\newcommand{\IH}{{\rm IH}}
\newcommand{\Gr}{{\rm Gr}}

\newcommand{\Sing}{{\rm Sing}}

\newcommand{\al}{\alpha}

\newcommand{\Ga}{\Gamma}

\newcommand{\De}{\Delta}

\newcommand{\si}{\sigma}

\newcommand{\om}{\omega}
\newcommand{\Om}{\Omega}

\newcommand{\ddd}{{\rm d}}
\newcommand{\tos}{\,{\to}\,}
\newcommand{\eq}{\,{=}\,}
\newcommand{\defs}{\,{:=}\,}
\newcommand{\nes}{\,{\ne}\,}
\newcommand{\ins}{\,{\in}\,}
\newcommand{\sst}{\,{\subset}\,}
\newcommand{\stm}{\,{\setminus}\,}
\newcommand{\gess}{\,{\ges}\,}
\newcommand{\less}{\,{\les}\,}
\newcommand{\sgt}{\,{>}\,}

\newcommand{\col}{\,{:}\,}
\newcommand{\pl}{\one {+}\one}
\newcommand{\mi}{\one {-}\one}
\newcommand{\bl}{\bigl}
\newcommand{\br}{\bigr}

\newcommand{\ssb}{\raise.15ex\h{${\scriptscriptstyle\bullet}$}}
\newcommand{\ssc}{\,\raise.15ex\h{${\scriptstyle\circ}$}\,}

\newcommand{\into}{\hookrightarrow}
\newcommand{\simto}{\,\,\rlap{\hskip1.5mm\raise1.4mm\hbox{$\sim$}}\hbox{$\longrightarrow$}\,\,}

\makeatletter
\renewcommand\section{\@startsection{section}{1}{0pt}{-3ex plus -1ex minus -.2ex}{2.3ex plus.2ex}{\centering\normalfont\bfseries}}
\makeatother
\theoremstyle{plain}
\newtheorem{thm}{Theorem}[section]

\newtheorem{ithm}{Theorem}
\newtheorem{icor}{Corollary}

\theoremstyle{definition}
\newtheorem{rem}{Remark}[section]

\newtheorem{exam}{Example}[section]

\begin{document}
\title[Defect of projective hypersurfaces]{Defect of projective hypersurfaces\\with isolated singularities}
\author[S.-J. Jung]{Seung-Jo Jung}
\address{S.-J. Jung : Department of Mathematics Education, and Institute of Pure and Applied Mathematics, Jeonbuk National University, Jeonju, 54896, Korea}
\email{seungjo@jbnu.ac.kr}
\author[M. Saito]{Morihiko Saito}
\address{M. Saito : RIMS Kyoto University, Kyoto 606-8502 Japan}
\email{msaito@kurims.kyoto-u.ac.jp}
\thanks{This work was partially supported by National Research Foundation of Korea NRF-2021R1C1C1004097.}
\begin{abstract} Let $X$ be a hypersurface with isolated singularities defined by $f$ in ${\bf P^{n+1}}$ with $n>1$. The difference ${\rm def}(X):=h^{n+1}(X)-h^{n-1}(X)$ is called the defect of $X$ (for self-duality of the cohomology of $X$). It is known that its vanishing is closely related to ${\bf Q}$-factoriality of $X$ without assuming rational singularities when $n=3$. This number coincides with the dimension of the cokernel of the inclusion $H^{n-1}(X)\to{\rm IH}^{n-1}(X)$, the rank of the morphism from the vanishing cohomologies of $X$ to $H^{n+1}(X)$ for a one-parameter smoothing of $X$ with total space smooth, and also with the dimension of the unipotent monodromy part of the Milnor fiber cohomology of $f$ with degree $n$. In the case $X$ has only weighted homogeneous isolated singularities, the defect ${\rm def}(X)$ is then given by the $E_2$-term of the spectral sequence of the double complex with differentials ${\rm d}f\wedge$ and $\rm d$ by the $E_2$-degeneration of the pole order spectral sequence. It can be calculated explicitly using a computer even for analogues of the Hirzebruch quintic threefold with more than one hundred ordinary double points found by B.\ van Geemen and J.\ Werner in a compatible way with their computation. We give also an example with ${\rm def}(X)>0$ and $|{\rm Sing}\,X|=1$ in the non-projective-cone case where $n=3$.
\end{abstract}
\maketitle

\section*{Introduction} \label{intr}
\nin
Let $X\sst\PP^{n+1}$ be a hypersurface with {\it isolated\one} singularities with $n\gess2$. The {\it defect\one} $\df(X)$ is the difference $h^{n+1}(X)\mi h^{n-1}(X)$, where $h^k(X)\defs\dim H^k(X)$ for $k\ins\Z$ (which might be called rather the Poincar\'e number). This measures partly a failure of self-duality. Note that the pairing on the middle cohomology may be degenerate and the total failure is given by the link cohomology, see Remark~\ref{R1.2} below. The cohomology groups in this paper are with $\Q$-coefficients unless otherwise stated.
\sk
Let $\{\X_c\}_{c\in\De}$ be a one-parameter smoothing of $X\eq\X_0$ in $\PP^{n+1}$ such that the total space and the $\X_c$ are smooth for $c\ins\De^*$, where $\De$ is a sufficiently small disk. This can be obtained by choosing a smooth hypersurface of the same degree intersecting $X$ only at smooth points. Applying also the weak Lefschetz theorem, we then get the equalities
\begin{equation} \label{1}
h^k(X)=h^k(\X_c)\,(=0\,\,\,\h{or}\,\,\,1)\,\,\,\,\h{for any}\,\,\,k\nes n,\,n{+}1,\,\,c\ins\De^*,
\end{equation}
(depending on the parity of $k$) together with the exact sequence of mixed Hodge structures
\begin{equation} \label{2}
0\to H^n(X)\to H^n(\X_{\infty})\buildrel{\rho}\over\to V\buildrel{\si}\over\to H^{n+1}(X)\to H^{n+1}(\X_{\infty})\to 0.
\end{equation}
Here $H^{\ssb}(\X_{\infty})$ denotes the limit mixed Hodge structure of the one-parameter degeneration, and $V\defs\mopl_x\,V_x$ with $V_x$ the vanishing cohomology at $x\ins\Sing\,X$, see \cite{St76}, \cite{St77}, \cite{mhm}.
\sk
Let $\IH^k(X)$ be the intersection cohomology. We have the canonical isomorphisms
\begin{equation} \label{3}
\iota^k:H^k(X)\simto\IH^k(X)\q\h{for any}\,\,\,k\nes n{-}1,\,n,
\end{equation}
together with the exact sequence of mixed Hodge structures
\begin{equation} \label{4}
0\to H^{n-1}(X)\buildrel{\iota^{n-1}}\over\longrightarrow\IH^{n-1}(X)\to{}^NV_1(1)\to H^n(X)\buildrel{\iota^n}\over\to\IH^n(X)\to 0.
\end{equation}
Here $V_1\sst V$ denotes the unipotent monodromy part, $^NV_1\defs{\rm Ker}\,N\sst V_1$, and $N\defs\log T_u$ with $T\eq T_sT_u$ the Jordan decomposition of the monodromy (and $(1)$ is a Tate twist).
\sk
Let $H^{\ssb}(\Ff)_1$ be the unipotent monodromy part of the Milnor fiber cohomology with $f$ a defining polynomial of $X$. The following theorem and corollaries generalize some assertions in the literature (see for instance \cite{NaSt}, \cite{vGeWe}).

\begin{ithm} \label{T1}
There are equalities
\begin{equation} \label{5}
\begin{aligned}
&\df(X)=\dim{\rm Coker}\,\iota^{n-1}\\&=\dim{\rm Coker}\,\rho=\dim{\rm Im}\,\si=\dim H^n(\Ff)_1.
\end{aligned}
\end{equation}
\end{ithm}

\begin{icor} \label{C1}
For $c\ins\De^*$ we have
\begin{equation} \label{6}
\begin{aligned}
&\dim\IH^n(X)\\&=\dim H^n(\X_c)-\dim V-\dim{}^NV_1+2\one\df(X).
\end{aligned}
\end{equation}
\end{icor}

\begin{icor} \label{C2}
Assume $n\eq3$ and $1$ is not a spectral number of any singular point of $X$ or equivalently $V_1$ is isomorphic to a direct sum of copies of $\Q(-2)$ $($for instance $X$ has only rational singularities$)$. Then for $c\ins\De^*$ we have
\begin{equation} \label{7}
\begin{aligned}
&\dim_{\C}\Gr_F^2\IH^3(X)_{\C}\\&{}=\dim_{\C}\Gr_F^2H^3(\X_c)_{\C}-\dim_{\C}\Gr_F^2V_{\C}+\df(X),
\end{aligned}
\end{equation}
where $E_{\C}$ means $\C{\otimes}_{\Q}E$ for a $\Q$-vector space $E$ in general.
\end{icor}

In \eqref{7} we have $\dim\Gr_F^2H^3(\X_c)_{\C}\eq n_{h,2}$, the multiplicity of the spectral number 2 (or 3) of a defining polynomial $h$ of $\X_c$ (in five variables), see \eqref{9}. The assertion about $\dim H^n(\Ff)_1$ in \eqref{5} would be well known to specialists although it might be forgotten to mention explicitly in the literature, see for instance \cite{Di}. The proofs of Corollaries~\ref{C1} and \ref{C2} follow from Theorem~\ref{T1} easily using the exact sequences \eqref{2} and \eqref{4} together with the {\it Grothendieck group\one} of mixed Hodge structures and also the {\it invariance\one} of $\dim\Gr_F^p$ under the passage to the limit mixed Hodge structure. (The last invariance follows from the assertion that the Hodge filtration can be extended to a filtration of the Deligne extension whose graded quotients are locally free even after a unipotent base change.) Note that $\Gr_F^2(^NV_1(1))_{\C}\eq\Gr_F^2({\rm Coker}\,\iota^{n-1})_{\C}\eq0$ in Corollary~\ref{C2}, and the equality obtained by applying $\Gr_F^1$ instead of $\Gr_F^2$ is equivalent to \eqref{7} employing the self-duality of mixed Hodge structures on $\IH^3(X)$, $H^3(\X_c)$, and $V_{\ne1}$ (for instance $\DD V_{\ne1}\cong V_{\ne1}(3)$, but $\DD V_1\cong V_1(4)$), where $V_{\ne1}$ denotes the {\it non-unipotent monodromy part\one} of $V$. Recall also that $X$ has a rational singularity at $x\ins\Sing\,X$ if and only if the minimal spectral number at $x$ is strictly bigger than 1 (see \cite{exp}), and we have the symmetry of spectral numbers at $x$ whose center is 2 in the case $n\eq3$. By Theorem~\ref{T1} and the exact sequence \eqref{2} we have for $p\ins\Z$ (for instance $p\eq2$ when $n\eq3$) the inequality
\begin{equation} \label{8}
\df(X)\ges\dim_{\C}\Gr_F^pV_{\C}-\dim_{\C}\Gr_F^pH^n(\X_c)_{\C}.
\end{equation}
Note that $H^n(\X_{\infty})_{\ne1}\simto V_{\ne1}$.
\sk
The Euler and primitive Hodge numbers $\chi(\X_c)$ and $\dim\Gr^p_FH^n_{\rm prim}(\X_c)$ for $c\ins\De^*$ can be obtained respectively as the coefficients of $t^{n+1}$ and $t^{(p+1)d}$ (using the Hodge symmetry) of the formal power series
\begin{equation} \label{9}
d\one t(1\pl t)^{n+2}/(1\pl d\one t),\q\bl((t\mi t^d)/(1\mi t)\br)^{n+2}\in\Q[[t]],
\end{equation}
where $d\eq\deg f$, see \cite{Hi0}, \cite{Na}, \cite{Gri} (and also \cite[Remark at the end of Section 2]{MSS}). If all the singularities of $X$ are ordinary double points with $n$ odd, we have $V\eq V_1\eq{}^NV_1$, and their dimension is $s\defs|\Sing\,X|$. The equality~\eqref{6} for $n\eq3$ is then equivalent to a formula in the introduction of \cite{vGeWe} (where the dependence of $\df(X)$ on the position of ordinary double points is shown, see also Examples~\ref{E3.1} and \ref{E3.2} below) via the decomposition theorem for the resolution of $X$ by blowups at singular points.
\sk
The Milnor fiber cohomology $H^n(\Ff)_1$ can be calculated by using the pole order spectral sequence whose converging filtration is the pole order filtration, and it degenerates at $E_2$ if the singularities of $X$ are isolated and weighted homogeneous, see \cite{wh}. It is essentially the spectral sequence of the double complex with differentials $\ddd f\wedge$, $\ddd$, and the dimensions of the $E_2$-terms can be calculated effectively using a computer, see \cite{DiSt2}, \cite{nwh}, and also Section~\ref{S3} below.
\sk
It is known that the vanishing of the topological defect $\df(X)$ is closely related to $\Q$-factoriality without assuming rational singularities when $n\eq3$, see \cite{Che}, \cite{NaSt}, \cite{PaPo}, \cite{MaVi}, and also \cite{JS}.
\sk
In Section 1 we review some basics of vanishing cycles and intersection complexes. In Section 2 we give the proof of Theorem~\ref{T1}. In Section 3 we explain how to calculate the defect for some examples.

\tableofcontents
\numberwithin{equation}{section}

\section{Vanishing cycles and intersection complexes} \label{S1}
\nin
In this section we review some basics of vanishing cycles and intersection complexes, see also \cite[2.3]{FPS}, \cite[2.2]{KLS}.
\sk
Let $q\col\X\eq\bigsqcup_{c\in\De}\,\X_c\to\De$ be a one-parameter degeneration in the introduction. We have a short exact sequence of (analytic) mixed Hodge modules on $X$\,:
\begin{equation} \label{1.1}
0\to\Q_{h,X}[n]\to\psi_{q,1}\Q_{h,\X}[n]\buildrel{\!\!\rm can}\over\longrightarrow\varphi_{q,1}\Q_{h,\X}[n]\to0,
\end{equation}
where $\Q_{h,\X}[n{+}1]$ is a polarizable Hodge module of weight $n{+}1$ whose underlying $\Q$-complex is $\Q_{\X}[n{+}1]$, and $\Q_{h,X}[n]\defs\Hc^{-1}i^*_X(\Q_{h,\X}[n{+}1])$ with $i_X\col X\eq\X_0\into\X$ the inclusion, see \cite{mhp}, \cite{mhm}. The last inclusion can be identified locally with the inclusion $X\into\PP^{n+1}$ at each singular point of $X$. (Note that $\psi_q[-1]$, $\varphi_q[-1]$ preserve mixed Hodge modules.)
\sk
Taking an appropriate compactification of a local Milnor fibration around each $x\ins\Sing\,X$, one can verify that the vanishing cycle Hodge module $\varphi_{q,1}\Q_{h,\X}[n{+}1]$ is identified with the direct sum of the direct images $(i_x)_*V_x$ with $i_x\col\{x\}\into X$ the inclusion. The equality \eqref{1} and the exact sequence \eqref{2} then follow from \eqref{1.1} by taking the direct image by $a_X\col X\to pt$, see \cite{mhm}.
\sk
The dual exact sequence of \eqref{1.1} is identified with 
\begin{equation} \label{1.2}
0\to\varphi_{q,1}\Q_{h,\X}(1)[n]\buildrel{\!\!\rm var}\over\longrightarrow\psi_{q,1}\Q_{h,\X}[n]\to\DD\bl(\Q_{h,X}(n)[n]\br)\to0,
\end{equation}
where $\DD$ is the dual functor, see \cite{dual}. Since the composition ${\rm var}\ssc{\rm can}$ coincides with
\begin{equation*}
N\col\psi_{q,1}\Q_{h,\X}[n]\to\psi_{q,1}\Q_{h,\X}(-1)[n].
\end{equation*}
we then get the isomorphisms
\begin{align}
&\Q_{h,X}[n]={\rm Ker}\,N,\q\DD\bl(\Q_{h,X}(n)[n]\br)=({\rm Coker}\,N)(1),\label{1.3}\\ &\varphi_{q,1}\Q_{h,\X}(1)[n]={\rm Coim}\,N. \label{1.4}
\end{align}

Let $W$ be the weight filtration on $\psi_{q,1}\Q_{h,\X}[n]$, $\varphi_{q,1}\Q_{h,\X}[n]$. This is given by the monodromy filtration shifted by $n$ and $n{+}1$ respectively so that there are isomorphisms
\begin{equation} \label{1.5}
\begin{aligned}
N^j:\Gr^W_{n+j}(\psi_{q,1}\Q_{h,\X}[n])&\simto\Gr^W_{n-j}(\psi_{q,1}\Q_{h,\X}[n])(-j),\\
N^j:\Gr^W_{n+1+j}(\varphi_{q,1}\Q_{h,\X}[n])&\simto\Gr^W_{n+1-j}(\varphi_{q,1}\Q_{h,\X}[n])(-j),
\end{aligned}
\end{equation}
These imply the $N$-primitive decompositions
\begin{equation} \label{1.6}
\begin{aligned}
\mopl_j\Gr^W_j(\psi_{q,1}\Q_{h,\X}[n])&=\mopl_{j\ges0}\mopl_{k=0}^j\,N^k{}_P\Gr^W_{n+j}(\psi_{q,1}\Q_{h,\X}[n])(k),\\
\mopl_j\Gr^W_j(\varphi_{q,1}\Q_{h,\X}[n])&=\mopl_{j\ges0}\mopl_{k=0}^j\,N^k{}_P\Gr^W_{n+1+j}(\varphi_{q,1}\Q_{h,\X}[n])(k),
\end{aligned}
\end{equation}
where the $N$-primitive part is defined by
\begin{equation*}
{}_P\Gr^W_{n+j}(\psi_{q,1}\Q_{h,\X}[n])\defs{\rm Ker}\,N^{j+1}\sst\Gr^W_{n+j}(\psi_{q,1}\Q_{h,\X}[n]),
\end{equation*}
and similarly for ${}_P\Gr^W_{n+1+j}(\varphi_{q,1}\Q_{h,\X}[n])$ with $n$ replaced by $n{+}1$. Combined with \eqref{1.4}, we then get the isomorphism of mixed Hodge modules
\begin{equation} \label{1.7}
{}_P\Gr^W_{n+j}(\psi_{q,1}\Q_{h,\X}[n])=\begin{cases}{}_P\Gr^W_{n+j}(\varphi_{q,1}\Q_{h,\X}[n])&\h{if}\q j\sgt0\\\IC_X\Q_h&\h{if}\q j\eq0,\end{cases}
\end{equation}
where $\IC_X\Q_h$ denotes the intersection complex Hodge module. The case $j\eq0$ follows from \eqref{1.3} and \cite[(4.5.9)]{mhm} (using algebraic mixed Hodge modules, since $X$ is algebraic), see also \cite[2.3]{FPS}, \cite[2.2]{KLS}. 
\sk
Combined with \eqref{1.3} and \eqref{1.4}, these imply the following (which is well known for the underlying $\Q$-complexes).

\begin{thm} \label{T1.1}
We have the short exact sequence of mixed Hodge modules
\begin{equation} \label{1.8}
0\to\mopl_x\,(i_x)_*{}^NV_{x,1}(1)\into\Q_{h,X}[n]\to\IC_X\Q_h\to0,
\end{equation}
where $\mopl_x{}^NV_{x,1}={}^NV_1$.
\end{thm}

\begin{proof}
The kernel filtration $K_i$ on $\psi_{q,1}\Q_{h,\X}[n]$ is defined by ${\rm Ker}\,N^{i+1}$ so that we have the inclusion $N:\Gr^K_1\into\Gr^K_0$ giving the one in \eqref{1.8}.
\end{proof}

\begin{rem} \label{R1.1}
The isomorphisms in \eqref{3} and the exact sequence \eqref{4} follow from \eqref{1.8}.
\end{rem}

\begin{rem} \label{R1.2}
The kernel of $\iota^n$ in \eqref{4} expresses the failure of the self-duality of $H^n(X)$ for the canonical self-pairing. So the link cohomology, which is isomorphic to $^NV_1$ (see \cite{Mi}), gives the failure of the self-duality of the cohomology of $X$.
\end{rem}

\begin{rem} \label{R1.3}
It is well known that a complex variety $X$ is a $\Q$-homology manifold if and only if the canonical morphism $\Q_X[\dim X]\tos\IC_X\Q$ is an isomorphism, where the last condition is equivalent to that the canonical morphism $\Q_X[\dim X]\tos\DD\bl(\Q_X(\dim X)[\dim X]\br)$ is a quasi-isomorphism (using \cite[(4.5.6--7 and 9)]{mhm}).
\end{rem}

\section{Proof of the main theorem} \label{S2}
\nin
In this section we give the proof of Theorem~\ref{T1}.
\msn
{\it Proof for\,} $\dim{\rm Coker}\,\iota^{n-1}$.
By self-duality of the intersection cohomology of $X$ (see for instance \cite{BBD}) we have
\begin{equation} \label{2.1}
\dim\IH^{n-1}(X)=\dim\IH^{n+1}(X).
\end{equation}
So the assertion follows from \eqref{3} and \eqref{4}.
\msn
{\it Proof for\,} $\dim{\rm Coker}\,\rho$ and $\dim{\rm Im}\,\si$.
The morphism $\si$ induces the quotient morphism
\begin{equation*}
\si':{}_NV_1\defs V_1/NV_1\to H^{n+1}(X)\eq\IH^{n+1}(X),
\end{equation*}
(see \eqref{3}), since $\si$ is compatible with the action of the monodromy. In view of the argument in Section~\ref{S1} it is induced by the short exact sequence of mixed Hodge modules
\begin{equation} \label{2.2}
0\to\IC_X\Q_h\into\DD\bl(\Q_{h,X}(n)[n]\br)\to\mopl_x\,(i_x)_*\one{}_NV_{x,1}\to0,
\end{equation}
which is the dual of \eqref{1.8}. So the assertion follows.

\msn
{\it Proof for\,} $\dim H^n(\Ff)_1$.
Set $Y\defs\PP^{n+1}$, $U\defs Y\stm X$. It is known that there are isomorphisms
\begin{equation} \label{2.3}
H^k(\Ff)_1=H^k(U)\q(\forall\,k\ins\Z),
\end{equation}
(see for instance \cite{Di}, \cite[1.3]{BuSa}) together with the local cohomology sequence
\begin{equation} \label{2.4}
\to H^k(Y)\to H^k(U)\to H_X^{k+1}(Y)\to H^{k+1}(Y)\to.
\end{equation}
We have moreover the isomorphisms
\begin{equation} \label{2.5}
H_X^{k+1}(Y)=H_{2n-k+1}(X)({-}n{-}1)\q(\forall\,k\ins\Z),
\end{equation}
since
\begin{equation} \label{2.6}
\RR\Ga_X\Q_Y\eq\RR\Ga_X(\DD\Q_Y)({-}n{-}1)[{-}2n{-}2]\eq(\DD\Q_X)({-}n{-}1)[{-}2n{-}2].
\end{equation}
So the assertion follows using \eqref{1}.

\section{Explicit computation of defects} \label{S3}
\nin
In this section we explain how to calculate the defect for some examples.
\ms
Let $\Om^{\ssb}$ be the complex of graded algebraic differential forms on $\C^{n+2}$ whose components are finite free graded modules over $R:=\C[x_0,\dots,x_{n+1}]$. Here the $x_i$ are the coordinates of $\C^{n+2}$, which have degree 1 as well as the $\ddd x_i$. For $k\in\Z$, we have the microlocal {\it pole order spectral sequence}
\begin{equation} \label{3.1}
E_1^{p,q}(f)_k=H^{p+q}_{\dfw}(\Om^{\ssb})_{qd+k}\Longrightarrow\Ht^{p+q-1}(\Ff,\C)_{\ee(-k/d)},
\end{equation}
converging to the pole order filtration $P$ up to shift, where $_{\ee(-k/d)}$ denotes the $\ee(-k/d)$-eigenspace of the monodromy with $\ee(\al)\defs e^{2\pi\sqrt{-1}\,\al}$ for $\al\ins\Q$, see \cite{KCS}, \cite{nwh}. Set
\begin{align*}
M:=H^{n+2}_{\dfw}(\Om^{\ssb}),&\q N:=H^{n+1}_{\dfw}(\Om^{\ssb})(-d),\\
M^{(2)}:=H^{n+2}_{\ddd}(H^{\ssb}_{\dfw}(\Om^{\ssb})),&\q N^{(2)}:=H^{n+1}_{\ddd}(H^{\ssb}_{\dfw}(\Om^{\ssb}))(-d).
\end{align*}
Here $(m)$ denotes the shift of grading by $m\in\Z$ so that $G(m)_k=G_{k+m}$ ($k\in\Z$) for a graded module $G$, and $M,N$ are also denoted as $M^{(1)}, N^{(1)}$. For $r\eq1,2$, we have
\begin{equation}
E_r^{p,q}(f)_k=\begin{cases} \label{3.2}
M^{(r)}_{qd+k}&\h{if}\,\,\,p\pl q\eq n\pl 2,\\
N^{(r)}_{(q+1)d+k}\raise15pt\h{}&\h{if}\,\,\,p\pl q\eq n\pl 1,\end{cases}
\end{equation}
and $E_r^{p,q}(f)_k\eq 0$ otherwise, since $\dim\Sing\,f^{-1}(0)\less 1$, see \cite{cons}. We have the following.

\begin{thm}[\cite{wh}] \label{T3.1}
If the singularities of $X$ are isolated and weighted homogeneous, this spectral sequence degenerates at $E_2$.
\end{thm}

We can verify that the rank of the differential $d_1\col N_{k+d}\tos M_k$ coincides with
\begin{equation} \label{3.3}
{\rm rk}\,\phi_k-({\rm rk}\,\Gr^G_0\phi_k\pl{\rm rk}\,\Gr^G_1\phi_k).
\end{equation}
Here the morphism
\begin{equation*}
V\defs\Om^{n+1}_{k-n-1}\oplus\Om^{n+1}_{k-d-n-1}\buildrel{\phi_k}\over\longrightarrow V'\defs\Om^{n+2}_{k+d-n-2}\oplus\Om^{n+2}_{k-n-2}
\end{equation*}
is defined by
\begin{equation*}
\phi_k(\om,\om')\defs(\dfw\one\om,\,\ddd\om+\dfw\one\om'),
\end{equation*}
and the increasing filtration $G$ is given so that the $\Gr^G_i\phi_k$ for $i\eq0,1$ is induced by $\dfw$.

\begin{exam}\label{E3.1}
A sample code for Singular \cite{Sing} to calculate the dimension of $N^{(2)}_{3d}$ is as below. Here the hypothesis of Corollary~\ref{C2} is practically assumed in order that this dimension coincide with the defect (although this can be avoided by modifying the code so that $\dim N^{(2)}_{2d}$ is also computed, where the Hodge symmetry and \cite[Proposition 2.2]{DiSt2} are used). We first determine the coefficients of the input polynomial on the first line, which comes from \cite[Proposition 2.3]{vGeWe} and is an analogue of the one in \cite{Hi}.
\ms
\vbox{\tiny\tt\pv@LIB "elim.lib"; ring R=0,(x,y,z,u,v),dp; poly f,g,h; g=(x+6*z)*(y^2-x^2)*(5*y^2-4*(x+z)^2);@
\pv@f=g-subst(g,x,u,y,v); int a1,a2,a3,b1,b2,b3,i,j,k,l,d,p,q,m,n,s,k0,k1,maxn,rk1,rk2,rk3,gam;@
\pv@intmat F[200][6]; intmat psum[2][2]; intmat siz[2][2]; d=deg(f); n=0; for(i=0; i<=d; i++)@
\pv@{for(j=0;j<=d-i;j++){for(k=0;k<=d-i-j;k++){for(l=0;l<=d-i-j-k;l++){g=f; for(p=1;p<=i;p++)@
\pv@{g=diff(g,x)/p;} for(p=1; p<=j; p++){g=diff(g,y)/p;} for(p=1; p<=k; p++){g=diff(g,z)/p;}@
\pv@for(p=1; p<=l; p++){g=diff(g,u)/p;} g=subst(g,x,0,y,0,z,0,u,0,v,1); if (g!=0)@
\pv@{printf("
\pv@F[n,4]=l; F[n,5]=d-i-j-k-l; F[n,6]=int(g);}}}}} maxn=n;@}
\sk
\vbox{\tiny\tt\pv@int nn; list I; proc rn() {nn=(I[1]-I[2]+3)*(I[1]-I[2]+2)*(I[1]-I[2]+1)*(I[1]-I[2]) div@
\pv@24+ (I[1]-I[2]-I[3]+2)*(I[1]-I[2]-I[3]+1)*(I[1]-I[2]-I[3]) div 6+(I[1]-I[2]-I[3]-I[4]+1)*@
\pv@(I[1]-I[2]-I[3]-I[4]) div 2+ I[1]-I[2]-I[3]-I[4]-I[5]+1;}@
\pv@g=((x^d-x)/(x-1))^5; for(p=1; p<=2*d; p++) {g=diff(g,x)/p;} gam=int(subst(g,x,0));@
\pv@int mu2, mu3, nu3; intmat siz[2][2]; intmat psum[2][3]; psum[1,1]=0; psum[2,1]=0;@
\pv@for (s=1; s<=2; s++) {k0=s*d-4; k1=(s+1)*d-5; siz[1,s]=(k0+1)*(k0+2)*(k0+3)*(k0+4) div@
\pv@24; siz[2,s]=(k1+1)*(k1+2)*(k1+3)*(k1+4) div 24; psum[1,s+1]=psum[1,s]+siz[1,s];@
\pv@psum[2,s+1]=psum[2,s]+siz[2,s];} a1=5*siz[1,1]; a2=5*siz[1,2]; a3=5*psum[1,3]; b1=@
\pv@siz[2,1]; b2=siz[2,2]; b3=psum[2,3]; intmat A[a1][b1]; intmat B[a2][b2]; intmat C[a3][b3];@}
\msn
Here the procedure (or function) rn is for numbering of monomials of a given degree. We then calculate the coefficients of the matrix corresponding to $\phi_{2d}$.
\ms
\vbox{\tiny\tt\pv@for (s=1; s<=2; s++) {k0=s*d-4; for (i = 0; i <= k0; i++) {for (j = 0; i+j <= k0;@
\pv@j++) {for (k = 0; i+j+k <= k0; k++) {for (l = 0; i+j+k+l <= k0; l++) {I=k0,i,j,k,l;@
\pv@rn(); p=nn+5*psum[1,s]; k1=(s+1)*d-5; m=siz[1,s]; for (n = 1; n <= maxn; n++) {if@
\pv@(F[n,1] > 0) {I=k1,i+F[n,1]-1,j+F[n,2],k+F[n,3],l+F[n,4]; rn(); q=nn+psum[2,s]; C[p,q]=@
\pv@C[p,q]+F[n,1]*F[n,6];} if (F[n,2] > 0) {I=k1,i+F[n,1],j+F[n,2]-1,k+F[n,3],l+F[n,4];@
\pv@rn(); q=nn+psum[2,s]; C[p+m,q]=C[p+m,q]+F[n,2]*F[n,6];} if (F[n,3] > 0) {I=k1,i+F[n,1],j+@
\pv@F[n,2],k+F[n,3]-1,l+F[n,4]; rn(); q=nn+psum[2,s]; C[p+2*m,q]=C[p+2*m,q]+F[n,3]*F[n,6];}@
\pv@if (F[n,4] > 0) {I=k1,i+F[n,1],j+F[n,2],k+F[n,3],l+F[n,4]-1; rn(); q=nn+psum[2,s];@
\pv@C[p+3*m,q]=C[p+3*m,q]+F[n,4]*F[n,6];} if (F[n,5] > 0) {I=k1,i+F[n,1],j+F[n,2],k+F[n,3],@
\pv@l+F[n,4]; rn(); q=nn+psum[2,s]; C[p+4*m,q]=C[p+4*m,q]+F[n,5]*F[n,6];}}@
\pv@if (s > 1) {k1=s*d-5; if (i > 0) {I=k1,i-1,j,k,l; rn(); q=nn; C[p,q]=C[p,q]+i;} if@
\pv@(j > 0) {I=k1,i,j-1,k,l; rn(); q=nn; C[p+m,q]=C[p+m,q]+j;} if (k > 0) {I=k1,i,j,k-1,l;@
\pv@rn(); q=nn; C[p+2*m,q]=C[p+2*m,q]+k;} if (l > 0) {I=k1,i,j,k,l-1; rn(); q=nn;@
\pv@C[p+3*m,q]=C[p+3*m,q]+l;} if (i+j+k+l < k0) {I=k1,i,j,k,l; rn(); q=nn; C[p+4*m,q]=@
\pv@C[p+4*m,q]+k0-i-j-k-l;}}}}}}} for(i=1; i<=a1; i++) {for(j=1; j<=b1; j++) {A[i,j]=@
\pv@C[i,j];}} for(i=1; i<=a2; i++) {for(j=1; j<=b2; j++) {B[i,j]=C[i+a1,j+b1];}}@}
\msn
We now compute the rank of the matrices together with the total Tjurina number $\tau$, etc.
\ms
\vbox{\tiny\tt\pv@sprintf("Calculating the rank of a matrix of size 
\pv@sprintf("Calculating the rank of a matrix of size 
\pv@sprintf("Calculating the rank of a matrix of size 
\pv@mu3=(3*d-1)*(3*d-2)*(3*d-3)*(3*d-4) div 24-rk2; nu3=mu3-gam; mu2=(2*d-1)*(2*d-2)*(2*@
\pv@d-3)*(2*d-4) div 24-rk1; sprintf("mu_2=
\pv@mu2, mu3, nu3, rk3-rk2-rk1, nu3-rk3+rk2+rk1);@
\pv@d=deg(f); h=x^d+y^d+z^d+u^d+v^d; ideal J=jacob(h); p=size(kbase(std(J),2*d-5));@
\pv@J=jacob(f); s=size(kbase(std(J),5*d-9)); sprintf("n_{h,2}=
\msn
Here $\dim\Gr^2_FH^3(\X_c)_{\C}\eq101$, $s\eq 118$, $\df(X)\eq19$, where $c\ins\De^*$ (this is the same for the examples below). Note that mu2, mu3, nu3 mean $\mu_2$, $\mu_3$, $\nu_3$ respectively, and {\scriptsize\tt\verb@n_{h,2}@} is the coefficient of $t^2$ in the spectrum ${\rm Sp}_h(t)$ of a defining polynomial $h$ of $\X_c$. Here one can calculate the defect also by using the code:
\ms
\vbox{\tiny\tt\pv@ideal K=sat(J,maxideal(1))[1]; sprintf("def(X)=
\msn
(This does not work with {\small\tt\verb@ds;@} instead of {\small\tt\verb@dp;@}.) Adding this, the first and last two lines of the above code for Singular are enough {\it only\one} in the ordinary double point case (where $\tau\eq s$), see Remark~\ref{R3.2} below. One can see whether the last condition is satisfied by using the following.
\ms
\vbox{\tiny\tt\pv@ideal L=radical(J); size(kbase(std(L),5*d-9));@}
\msn
If the output coincides with {\small\tt\verb@tau@}, it is verified.
\end{exam}

\begin{exam}\label{E3.2}
For a polynomial coming from \cite[Proposition 2.1]{vGeWe}, one may replace $g$ with
\sk
{\footnotesize\tt\pv@g=(x+z)*(3*y^2-(x-2*z)^2)*(5*x^2+5*y^2-8*z^2);@}
\skn
Here $\dim\Gr^2_FH^3(\X_c)_{\C}\eq101$, $s\eq 118$, but $\df(X)\eq18$.
\end{exam}

\begin{exam}\label{E3.3}
For a polynomial in \cite{vSt}, one can replace $f$ with
\sk
{\footnotesize\tt\pv@poly w=-x-y-z-u-v; f=-(x^5+y^5+z^5+u^5+v^5+w^5)+@}

{\footnotesize\tt\pv@10*(x*y*z*u*v+x*y*z*u*w+x*y*z*v*w+x*y*u*v*w+x*z*u*v*w+y*z*u*v*w);@}
\skn
Here $\dim\Gr^2_FH^3(\X_c)_{\C}\eq101$, $s\eq 130$, $\df(X)\eq29$.
\end{exam}

\begin{exam}\label{E3.4}
For a polynomial in \cite{Che}, one may put
\sk
{\footnotesize\tt\pv@int b=4; f=x*(x^b+y^b+z^b+u^b+v^b)+y*(x^b-2*y^b+3*z^b-4*u^b+5*v^b);@}
\skn
Here $\dim\Gr^2_FH^3(\X_c)_{\C}\eq101$, $s\eq 16$, $\df(X)\eq1$.
\end{exam}

\begin{exam}\label{E3.5}
For the Segre cubic hypersurface, one can set
\sk
{\footnotesize\tt\pv@f=(x+y+z+u+v)^3-(x^3+y^3+z^3+u^3+v^3);@}
\skn
Here $\dim\Gr^2_FH^3(\X_c)_{\C}\eq5$, $s\eq 10$, $\df(X)\eq5$.
\end{exam}

\begin{exam}\label{E3.6}
For an example with $\df(X)\sgt0$ and $|\Sing\,X|\eq1$ (see also \cite{Li}), put
\sk
{\footnotesize\tt\pv@f=(y^2-2*x*z)^2+(y^2-2*x*z)*x^2+x^4+u^4+v^4;@}
\skn
Here $\dim\Gr^2_FH^3(\X_c)_{\C}\eq30$, $\dim V_1\eq\df(X)\eq7$, and $X$ has only one singular point, which is defined analytic-locally by $y^2\pl x^8\pl u^4\pl v^4\eq0$, see also \cite[Remark 8.5]{wh}. This is a rational singularity, since the minimal spectral number is $\tfrac{9}{8}$. (It seems quite difficult to find a non-rational example for $d\eq4$.)
\end{exam}

\begin{rem}\label{R3.1}
For an example with $d\eq6$ (or Example~\ref{E3.3}) it may be better to use the following code written by C. (Note that this uses an experimental method to calculate the rank of a matrix via mod p reduction.) It must be copied to a text file with name c.c for instance, and must be compiled by gcc, etc. Here Return must be added after {\small\tt\verb@.h>@} (two places). During the copying some character might be {\it modified,} for instance {\small\tt\verb@'@} or {\small\tt\verb@^@} might be affected depending on the viewer, and {\it this should be replaced by a character from a keyboard using a text editor.}
\ms
\vbox{\tiny\tt\pv@#include<stdio.h> #include<stdlib.h> static const int prims[] = {32633, 32647, 32653, 32687,@
\pv@32693, 32707, 32713, 32717, 32719, 32749}; int A[3000][3000], B[3000][3000], F[300][7],@
\pv@siz[2][5], psum[2][6]; int g1[5*15], g2[5*15], g3[5*15], gam[5*15]; int c, d, i, j, k, k0,@
\pv@k1, l, m, rk1, rk2, rk3, rk4, maxn, Prim, mu2, mu3, nu3; FILE *fp1; void rank(int a, int b,@
\pv@int b2); int rn(int mm, int ii, int jj, int kk, int ll); void nextc(void); int main(void)@
\pv@{int a, b, cc, n, p, q, r, s; if ((fp1 = fopen("Inp.txt", "r")) == NULL) printf@
\pv@("Input file name must be Inp.txt\n"), exit(1); printf("Prime Number "); cc=getchar(); if@
\pv@('0'<=cc && cc<='9') Prim=prims[9-(cc-'0')]; else Prim=prims[9]; printf("Prime 
\sk
\vbox{\tiny\tt\pv@for (i=j=0; c!=EOF && c!='/' && i<=300;) {nextc(); if (c=='-') {F[i][6]=-1; c=fgetc(fp1);@
\pv@nextc();} if ('0'<=c && c<='9') {for (F[i][j]=0; '0'<=c && c<='9'; c=fgetc(fp1)) F[i][j]=@
\pv@F[i][j]*10+c-'0'; nextc(); j++;} if (i==0 && j==6) d=F[i][1]+F[i][2]+F[i][3]+F[i][4]+F[i][5];@
\pv@if (j==6 && F[i][1]+F[i][2]+F[i][3]+F[i][4]+F[i][5]!=d && i>0) {printf("Degree error\n");@
\pv@exit(1);} if (j==6) {if (F[i][6]<0) F[i][0]=F[i][0]*F[i][6]; i++; j=0;}} if (i>300) printf@
\pv@("Too many monomials\n"), exit(2); maxn=i; fclose(fp1); for (i=0; i<=5*d; i++) {if (i>0&&i<d)@
\pv@g1[i]=1; else g1[i]=0;} for (i=0; i<=5*d; i++) for (g2[i]=j=0; j<=i; j++) g2[i]=g2[i]+g1[j]*@
\pv@g1[i-j]; for (i=0; i<=5*d; i++) for (g3[i]=j=0; j<=i; j++) g3[i]=g3[i]+g1[j]*g2[i-j]; for@
\pv@(i=0; i<=5*d; i++) for (gam[i]=j=0; j<=i; j++) gam[i]=gam[i]+g2[j]*g3[i-j]; psum[0][0]=0;@
\pv@psum[1][0]=0; for (s=0; s<=1; s++) {k0=(s+1)*d-4; k1=(s+2)*d-5; if (k0>=0) siz[0][s]=(k0+1)*@
\pv@(k0+2)*(k0+3)*(k0+4)/24; else siz[0][s]=0; if (k1>=0) siz[1][s]=(k1+1)*(k1+2)*(k1+3)*(k1+4)/@
\pv@24; else siz[1][s]=0; psum[0][s+1]=psum[0][s]+siz[0][s]; psum[1][s+1]=psum[1][s]+siz[1][s];}@
\pv@a=5*psum[0][1+1]; b=psum[1][1+1];@}
\sk
\vbox{\tiny\tt\pv@for (s=0; s<=1; s++) {k0=(s+1)*d-4; for (i = 0; i <= k0; i++) {for (j = 0; i+j <= k0; j++)@
\pv@{for (k = 0; i+j+k <= k0; k++) {for (l = 0; i+j+k+l <= k0; l++) {p=rn(k0,i,j,k,l)+5*@
\pv@psum[0][s]; k1=(s+2)*d-5; m=siz[0][s]; for (n = 0; n < maxn; n++) {if (F[n][1] > 0) {q=rn(k1,@
\pv@i+F[n][1]-1,j+F[n][2],k+F[n][3],l+F[n][4])+psum[1][s]; A[p][q]=A[p][q]+F[n][1]*F[n][0];} if@
\pv@(F[n][2] > 0) {q=rn(k1,i+F[n][1],j+F[n][2]-1,k+F[n][3],l+F[n][4])+psum[1][s]; A[p+m][q]=@
\pv@A[p+m][q]+F[n][2]*F[n][0];} if (F[n][3] > 0) {q=rn(k1,i+F[n][1],j+F[n][2],k+F[n][3]-1,l+@
\pv@F[n][4])+psum[1][s]; A[p+2*m][q]= A[p+2*m][q]+F[n][3]*F[n][0];} if (F[n][4] > 0) {q=rn(k1,@
\pv@i+F[n][1],j+F[n][2],k+F[n][3],l+F[n][4]-1)+ psum[1][s]; A[p+3*m][q]= A[p+3*m][q]+F[n][4]*@
\pv@F[n][0];} if (F[n][5] > 0) {q=rn(k1,i+F[n][1],j+F[n][2],k+F[n][3],l+F[n][4])+ psum[1][s];@
\pv@A[p+4*m][q]= A[p+4*m][q]+F[n][5]*F[n][0];}} if (s > 0) {k1=(s+1)*d-5; if (i > 0) {q=rn(k1,@
\pv@i-1,j,k,l)+psum[1][s-1]; A[p][q]=A[p][q]+i;} if (j > 0) {q=rn(k1,i,j-1,k,l)+psum[1][s-1];@
\pv@A[p+m][q]=A[p+m][q]+j;} if (k > 0) {q=rn(k1,i,j,k-1,l)+psum[1][s-1]; A[p+2*m][q]=@
\pv@A[p+2*m][q]+k;} if (l > 0) {q=rn(k1,i,j,k,l-1)+psum[1][s-1]; A[p+3*m][q]=A[p+3*m][q]+l;} if@
\pv@(i+j+k+l < k0) {q=rn(k1,i,j,k,l)+psum[1][s-1]; A[p+4*m][q]=A[p+4*m][q]+k0-i-j-k-l;}}}}}}}@}
\sk
\vbox{\tiny\tt\pv@for(i=0; i<5*siz[0][0]; i++) for(j=0; j<siz[1][0]; j++) B[i][j]=A[i][j]; printf@
\pv@("calculating a matrix of size 
\pv@("
\pv@rk1-rk2; rk4=rk2; for(i=0; i<5*siz[0][0]; i++) for(j=0; j<siz[1][0]; j++) A[i][j]=B[i][j];@
\pv@printf("calculating a matrix of size 
\pv@siz[1][0],0); printf("
\pv@("mu_2=
\sk
\vbox{\tiny\tt\pv@void rank(int a, int b, int b2) {int i, i0, j, p, rank, minc, numnz, minnz, tmp[3000]; if@
\pv@(a>=3000 || b>=3000) {printf("Too big matrix!\n"), exit(1);} rank=rk2=0; while (b>0) {for@
\pv@(i=minc=minnz=0, i0=-1; i<a; i++) {if (A[i][b-1]!=0) {if (i0==-1 || abs(A[i][b-1])<=minc)@
\pv@{for (p=numnz=0; p<b-1 && numnz<minnz; p++) {if (A[i][p]!=0) numnz++;} if (i0==-1 || abs(@
\pv@A[i][b-1])<minc || numnz<minnz) {i0=i; minnz=numnz; minc=abs(A[i][b-1]);}}}} if (i0==-1)@
\pv@{b--;} else {for (j=0; j<b; j++) tmp[j]=A[i0][j]; for (j=0; j<b-1; j++) {for (i=0; i<i0;@
\pv@i++) A[i][j] = (A[i][j]*tmp[b-1]-A[i][b-1]*tmp[j]) 
\pv@A[i-1][j]=(A[i][j]*tmp[b-1]-A[i][b-1]*tmp[j]) 
\pv@rank;} rk1=rank;} int rn(int mm, int ii, int jj, int kk, int ll) {return (mm-ii+3)*(mm-ii+@
\pv@2)*(mm-ii+1)*(mm-ii)/24+ (mm-ii-jj+2)*(mm-ii-jj+1)*(mm-ii-jj)/6+ (mm-ii-jj-kk+1)*(mm-ii-@
\pv@jj-kk)/2+ mm-ii-jj-kk-ll;} void nextc(void) {while ((c<'0'||c>'9') && c!='-' && c!='/' &&@
\pv@c!=EOF) c=fgetc(fp1);}@}
\msn
Before running ./a.out, one has to prepare the input file Inp.txt by running the first eight lines of the above program for Singular, where the output on the screen (starting from \h{3 (0,0,0,2,4)}) should be copied to Inp.txt using a text editor. This must be put in the same directory as a.out. When one runs ./a.out, the computer will demand to choose a prime. Here any number between 0 and 9 should be typed together with Return, and one can press only Return in the case the prime corresponding to 0 is chosen. Concerning $f$ in the code for Singular, it can be replaced for instance by
\sk
{\scriptsize\tt\pv@int b=2; int c=3; int a=b*c; poly f=(x^b+y^b+z^b+u^b+v^b)^c-(x^a+y^a+z^a+u^a+v^a);@}
\skn
Here we have $\dim\Gr^2_FH^3(\X_c)_{\C}\eq255$, $s\eq 285$, $\df(X)\eq40$, and $X$ has only ordinary double points. (The computation may take several minutes.) It does not calculate $\dim\Gr^2_FH^3(\X_c)_{\C}$ nor $s$, where the last two lines of the above code for Singular can be hired. If we set
\sk
{\scriptsize\tt\pv@int b=3; int c=2; int a=b*c; poly f=3*(x^a+y^a+z^a+u^a+v^a)-(x^b+y^b+z^b+u^b+v^b)^c;@}
\skn
we have $\dim\Gr^2_FH^3(\X_c)_{\C}\eq255$, $\dim V_1\eq\tfrac{1}{2}\dim V\eq180$, $\df(X)\eq30$ with ninety singular points, which are analytic-locally defined by $x^3\pl y^3\pl u^2\pl v^2\eq0$ with Milnor number 4.
\end{rem}

\begin{rem}\label{R3.2}
In the case $X$ has only {\it ordinary double points\one} as singularities with $n\defs\dim X$ {\it odd,} it is known (see \cite[Ch.\,6, Theorem 4.5]{Di}, \cite[Theorem 4.1]{DiSt1}, \cite[Proposition 3.7]{Kl}) that
\begin{equation} \label{3.4}
{\rm def}(X)=|\Sing\,X|-\mu''_{(n'-1)d/2}\,\,\bl(=\nu_{(n'+1)d/2}\br),
\end{equation}
in the notation of \cite{KCS}, \cite{nwh}, \cite{wh} (using \eqref{5} with $\tau\eq|\Sing\,X|$), where $n$ in these papers coincides with $n'\defs n{+}2$ in our paper. Note that the right-hand side of \eqref{3.4} coincides with the defect of linear system in \cite{Di}. The equality \eqref{3.4} is equivalent to the {\it vanishing\one} of the $E_1$-differential $\ddd_1\col N_{(n'+1)d/2}\tos M_{(n'-1)d/2}$ of the pole order spectral sequence using Theorem~\ref{T1} and \eqref{2.3} together with \cite[Corollary 2]{KCS}. The latter can be shown by employing \cite[Theorem 2.2 (where $q\eq m\mi1$ but not $m$) and also (1.1.3)]{DSW}. These are compatible also with computations of the above examples applying the code from \cite[Appendix]{wh}.
\end{rem}

\end{document}